\documentclass[11pt]{amsart}
\usepackage{amsmath,amssymb,amsthm,upref,graphicx,mathrsfs,mathabx}
\usepackage[top=35mm, bottom=35mm, left=35mm, right=35mm]{geometry}
\usepackage{color,xcolor}
\usepackage[
  colorlinks=true,
  linkcolor=blue,
  citecolor=blue,
  urlcolor=blue]{hyperref}

\numberwithin{equation}{section}
\newtheorem{theorem}{Theorem}[section]
\newtheorem{proposition}[theorem]{Proposition}
\newtheorem{corollary}[theorem]{Corollary}
\newtheorem{lemma}[theorem]{Lemma}
\newtheorem{remark}[theorem]{Remark}

\numberwithin{equation}{section}

\begin{document}
\title[Forward, vanishing, and reverse Bergman Carleson measures]{Characterization of forward, vanishing, and reverse Bergman Carleson measures using sparse domination}
\author{Hamzeh Keshavarzi }
\maketitle

\begin{abstract}
In this paper, using a new technique from harmonic analysis called sparse domination, we characterize the positive Borel measures including forward, vanishing, and reverse Bergman Carleson measures.
The main novelty of this paper is determining the reverse Bergman Carleson measures which have remained open from the work of Luecking [Am. J.
Math. 107  (1985) 85–111]. Moreover, in the case of forward and vanishing measures, our results extend the results of [J. Funct. Anal. 280 (2021), no. 6, 108897, 26 pp] from $1\leq p\leq q< 2p$ to all $0<p\leq q<\infty$. In a more general case, we characterize the positive Borel measures $\mu$ on $\mathbb{B}$ so that the radial differentiation operator $R^{k}:A_\omega^p(\mathbb{B})\rightarrow L^q(\mathbb{B},\mu)$ is bounded and compact.
Although we consider the
weighted Bergman spaces induced by two-side doubling weights, the results are new even on classical weighted Bergman spaces.
\end{abstract}

\section{Introduction}

Throughout the paper, $n$ is a fixed positive integer.
 Let $\mathbb{C}$ be the complex plane, $\mathbb{B}$ be the unit ball in $\mathbb{C}^n$, and $\mathbb{S}$ be the boundary of $\mathbb{B}$.
Let $H(\mathbb{B})$ be  the space of all analytic functions on $\mathbb{B}$.
A radial weight on $\mathbb{B}$ is a positive and integrable function $\omega$ on $\mathbb{B}$  such that $\omega(z)=\omega(|z|)$. If $0<p<\infty$, the weighted Bergman space is denoted by $A^p_\omega(\mathbb{B})$ and defined as:
 $$A^p_\omega= A^p_\omega(\mathbb{B})=\{ f\in H(\mathbb{B}): \ \|f\|_{\omega,p} =\Big{[} \int_\mathbb{B} |f(z)|^p \omega(z)dV(z)\Big{]}^{1/p}<\infty   \},$$
where  $V$ is the normalized Lebesgue measure on $\mathbb{B}$. Throughout the paper, we use the notations $dV_\omega=\omega dV$ and $\|.\|_\omega=\|.\|_{\omega,2}$.
For $p\geq 1$, $A^p_\omega$ is a Banach space and for $0<p<1$, it is a complete metric space. For every $\alpha>-1$, the classical weighted Bergman space, denoted by $A^p_{\alpha}$, is the space $A^p_{(1-|z|^2)^\alpha}$. For more details about of  Bergman spaces see \cite{pelaez1, zhu1}.

Let $\widehat{\mathcal{D}}$ be the class of all radial weights $\omega$ on $\mathbb{B}$ where $\widehat{\omega}(z)=\int_{|z|}^1 \omega(s)ds$ satisfies the doubling condition
$ \widehat{\omega}(r)\lesssim \widehat{\omega}(\frac{1+r}{2})$.  For some examples and details of these weights see \cite{pelaez1, pelaez5}.
We say that the weight $\omega\in \widehat{\mathcal{D}}$ is regular, denoted by $\omega\in \mathcal{R}$, if
 \begin{equation*}
\dfrac{\widehat{\omega}(r)}{(1-r) \omega(r)}\simeq 1, \ \  0\leq r<1,
 \end{equation*}
 and it is called rapidly increasing, denoted by $\omega\in \mathcal{I}$, if
  \begin{equation*}
\lim_{r\rightarrow 1} \dfrac{\widehat{\omega}(r)}{(1-r)\omega(r)}=\infty.
 \end{equation*}

   It is well known that $\mathcal{R}\bigcup\mathcal{I}\subseteq\widehat{\mathcal{D}}$. J. A. Peláez, J. Rättyä \cite{pelaez1} introduced the classes $\mathcal{R}$ and $\mathcal{I}$ and studied the weighted Bergman spaces $A_\omega^q(\mathbb{D})$ induced by these weights.  Some of the results of \cite{pelaez1} have been extended from $\omega\in\mathcal{R}\bigcup\mathcal{I}$ to $\omega\in\widehat{\mathcal{D}}$ in \cite{pelaez5}. Du et al. \cite{du1} studied these classes of weighted Bergman spaces on $\mathbb{B}$.

  Consider $\widecheck{\mathcal{D}}$ as the set of all radial weights $\omega$ on $\mathbb{B}$ so that there exist $K=K(\omega)>1$ and $C=C(\omega)>1$ such that
$$\widehat{\omega}(r)\geq C \widehat{\omega}(1-\frac{1-r}{K}), \ \ 0\leq r<1.$$
We know that $\mathcal{R}\subset \widehat{\mathcal{D}}\bigcap \widecheck{\mathcal{D}}$. In \cite[Proposition 6]{pelaez4}, it has been shown that the radial Bekoll\'{e}-Bonami weights are in $\widehat{\mathcal{D}}\bigcap\widecheck{\mathcal{D}}$.
 For a definition of Bekoll\'{e}-Bonami weights see \cite[Page 8]{pelaez1}. Bekoll\'{e} and Bonami introduced these weights in \cite{bekolle1,bekolle2}.

Let us explain the sparse domination technique. In recent decades, several studies have been conducted in the
field of weighted inequalities related to the precise obtaining of the
optimal bounds (in terms of the $A_p$ constant of the weights) of the weighted operator norm of Calder\'{o}n–Zygmund operators. A lot of work was done to reach the full proof of the $A_2$ theorem by T. Hyt\"{o}nen \cite{hytonen} from the proof of the linear dependence on the $A_2$ constant of $\omega$ of the $L^2(\omega)$
norm of the Ahlfors-Beurling transform \cite{petermichl}. For a survey of the advances on the topic, see \cite{hytonen, lerner} and the references therein.

After the work of Hyt\"{o}nen, A. Lerner \cite{lerner}, with an alternative proof of the $A_2$
theorem, showed that the norm of the Calder\'{o}n–Zygmund operators can be controlled
 from above by a very special dyadic type of operators called sparse operators.
Let $\mathcal{S}$ be a collection of dyadic cubes in a dyadic grid $\mathcal{D}$ (see Sect. 2 for the definition). The sparse operator $\mathcal{A_S}$ is defined as:
$$ \mathcal{A_S} f(z)=\sum_{Q\in \mathcal{S}} \langle f\rangle_Q \mathbf{1}_Q (z), $$
where $\mathbf{1}_Q$ is the characteristic function of the cube $Q$ and collection $\mathcal{S}$ satisfies the condition that there exists some $\gamma \in (0, 1)$ such that
$$\sum_{S^\prime \in ch_\mathcal{S}(S)} |S^\prime|\leq \gamma |S|,$$
for every $S\in \mathcal{ S}$. Here, $ch_\mathcal{S}(S)$ denotes the set of the children of a dyadic
cube $S$ in $\mathcal{S}$. In section 3, we show that the modulus of the $k$-th radial derivative of functions in $H(\mathbb{B})$ will be controlled by sparse operators from above.

Suppose that $\mu$ is a positive Borel measure on the domain $D\subset \mathbb{C}^n$. Let $A$ be a Banach space of holomorphic functions on $D$ where $A$
is contained in $L^q(D, \mu)$ for some $q> 0$. Then, $\mu$  is called a $(q,A)$-Carleson
measure if the inclusion map $i: A\rightarrow L^q(D,\mu)$ is bounded, that is, there exists a constant $C > 0$ such that
$$\int_D |f|^q d\mu <C \|f\|^q_A, \ \ f\in A.$$
A $(q,A)$-Carleson
measure is called a vanishing $(q,A)$-Carleson
measure if $i: A\rightarrow L^q(\mu)$ is compact and it is called a reverse $(q,A)$-Carleson
measure if there exists a constant $D> 0$ such that
$$\int_D |f|^q d\mu >D \|f\|^q_A, \ \ f\in A.$$
If $A=A_\omega^p$, then a $(q,A_\omega^p)$-Carleson
measure is called a $(p,q,\omega)$-Bergman Carleson measure and if $p=q$, it is called a $\omega$-Bergman Carleson measure.

Carleson \cite{carleson2} introduced the Carleson measures to study the structure of the Hardy spaces and solved the corona problem on the disc.
He proved that a finite positive
Borel measure $\mu$ is a $(p,H^p(\mathbb{D}))$-Carleson measure if and only if there exists a constant $C > 0$
such that $\mu (S(z)) \leq C(1-|z|)$ for all $z\in \mathbb{D}$. The Carleson square $S(z)$ will be defined in section 2.

Hastings \cite{hastings} (see also Oleinik \cite{oleinik} and Oleinik and Pavlov \cite{oleinik2}) proved a similar
characterization for the $(p,A^p(\mathbb{D}))$-Carleson measures: a finite positive Borel
measure $\mu$ is a $(p,A^p(\mathbb{D}))$-Carleson measure if and only if there exists a constant $C > 0$ such that  $\mu (S(z)) \leq C(1-|z|)^2$ for all $z\in \mathbb{D}$.

The sets $S(z)$ are clearly not invariant under the automorphisms of the disc. The first characterization for the Carleson measures of the Bergman spaces of the unit ball
$\mathbb{B}$ was given by Cima and Wogen \cite{cima2}. The sets that they used were again not invariant under automorphisms. Luecking \cite{luecking1} (see also \cite[Theorem 14, p. 62]{duren} and \cite{duren2}) proved that it is possible to give
a characterization for the Carleson measures of Bergman spaces of $\mathbb{B}$ by using the balls for
the Bergman (or Kobayashi, or pseudohyperbolic) distance which are invariant under automorphisms: a finite positive Borel measure $\mu$
is a $(p,A^p(\mathbb{B}))$-Carleson measure if and only if, for some (and hence all) $0 < r < 1$, there
is a constant $C_r > 0$ such that
$$\mu(B(z_0, r))\leq C_r V(B(z_0, r)),$$
for all $z_0 \in \mathbb{B}$, and the Bergman
metric ball $B(z_0, r)\subset \mathbb{B}$.

 If $\omega$ is a regular weight, then there are $-1<\alpha\leq\beta<\infty$ such that $A^p_\alpha\subseteq A^p_\omega\subseteq A^p_\beta$. Moreover,  the weights $(1-r)^\alpha$ are regular for $\alpha>-1$.
If $\omega\in \mathcal{I}$, we have the inclusion $H^p\subset A^p_\omega\subset A^p_\beta$ for all $\beta >-1$. In \cite{pelaez1}, it has been shown that if $\omega$ is regular then, as in the classical weighted Bergman spaces, we can characterize the Carleson measures of $A_\omega^q$ by Bergman metric balls. However, if $\omega\in \mathcal{I}$, we must determine these measures by the Carleson squares as in the Hardy spaces. The reason is that for $\omega\in \mathcal{R}$ and $r>0$, we have $V_\omega(S(z))\simeq V_\omega(B_n(z,r))$. However,  this equivalence does not hold for $\omega\in \mathcal{I}$.

In some recent papers, the sparse bounds of operators have entered the spaces of complex holomorphic functions. To the best knowledge of the author, the first work in this regard belonged to Aleman, Pott, and Reguera \cite{aleman} who proved a pointwise sparse domination estimate of the Bergman projection and studied the Sarason conjecture on the Bergman spaces. Later, Rahm, Tchoundja, and Wick \cite{wick1} gave some weighted estimates for the Berezin transforms and Bergman projections acting on weighted Bergman spaces.

 B. Hu et al. \cite{wick2}, using sparse domination, gave a new characterization for the boundedness and compactness of weighted composition operators from $A_\alpha^p$ to $A^q_\alpha$ for $\alpha>-1$, $1\leq p\leq q<2p$  on the upper half-plane and the unit ball. In \cite{hu}, these  results were extended to strictly pseudoconvex bounded domains. Checking the validity of the results of these two papers for the classical $(p,q, (1-|.|)^\alpha)$-Bergman Carleson measures is easy. In sections 4 and 5, these results will be extended to all $0< p\leq q<\infty$.  Moreover, in a more general case, we determine the positive Borel measures $\mu$ on $\mathbb{B}$ so that the radial differential operator $R^{(k)}:A_\omega^q\rightarrow L^q(\mu)$ is bounded and compact for $\omega\in \widehat{\mathcal{D}}$ and $0<p\leq q<\infty$.  J. A. Peláez, J. Rättyä \cite{pelaez2} characterized the boundedness of the differentiation operator $D^{(k)}:A_\omega^p(\mathbb{D})\rightarrow L^q(\mu)$. In a one-variable case, there is not much difference between operators $D^{(k)}$ and $R^{(k)}$.

The reverse $\omega$-Bergman Carleson measures have been characterized only in the case $d\mu=\mathbf{1}_G dV_\alpha$ on $A_\alpha^p$, where $G$ is a Borel subset of $\mathbb{B}$.  In \cite{luecking3} for $\mathbb{D}$ and in \cite{luecking4} for $\mathbb{B}$, Luecking showed that if $d\mu=\mathbf{1}_G dV_\alpha$, then $\mu$ is a reverse Bergman Carleson measure for $A_\alpha^p$ if and only if there are $r,\delta>0$ such that
$$\dfrac{V(G\bigcap B(z,r))}{V(B(z,r))}>\delta, \qquad z\in \mathbb{B}.$$
In \cite[Theorem 9]{korhonen}, the necessity of this result has been extended to doubling weights on $\mathbb{D}$: If $\omega\in \widehat{\mathcal{D}}$, $G$ is a Borel subset of $\mathbb{D}$, and $d\mu=\mathbf{1}_G dV_\omega$ is a reverse $\omega$-Bergman Carleson measure, then
$$\inf_{z \in \mathbb{D}} \dfrac{V_\omega(G\bigcap S(z))}{V_\omega(S(z))}>0.$$

Moreover, in \cite{luecking5}, Luecking has conducted another valuable study on the general reverse $(q,A_\alpha^q)$-Carleson measures. He gave a sufficient condition in \cite[Theorem 4.2]{luecking5} and a necessary condition in \cite[Theorem 4.3]{luecking5} for this class of measures.  \cite[Theorem 4.2]{luecking5} has been extended to the weighted Bergman spaces induced by doubling weights in \cite[Theorem 10]{korhonen}. In section 5, using the sparse domination technique, we give a complete characterization for the reverse $\omega$-Bergman Carleson measures when $\omega\in \mathcal{D}$. Furthermore, we present new versions of \cite[Theorems 4.2 and 4.3]{luecking5}.

Throughout the paper, we denote by $s^\prime$  the
conjugate of $s>1$, that is, $\frac{1}{s}+\frac{1}{s^\prime}=1$. If $\{A_i\}$ and $\{B_i\}$ are two collections of positive functions (sets), then $A_i\lesssim B_i$  means that there exists some positive constant $C$ such that $A_i\leq  CB_i$ ($A_i\subseteq  CB_i$) for
each $i$. In this case, we say that $A_i$ is, up to a constant, less than or equal to $B_i$. Moreover,
 we say that $\{A_i\}$ and $\{B_i\}$ are comparable if
$A_i\simeq B_i$, that is, $B_i\lesssim A_i\lesssim B_i$.

\section{Preliminaries}
The radial derivative of a holomorphic function $f:\mathbb{B}\rightarrow \mathbb{C}$ is defined as
$$Rf(z)=\sum_{i=1}^n z_i \dfrac{\partial f}{\partial z_i} (z).$$
The radial derivative is a very important concept of differentiation on the unit ball. The reason for this naming is that
$$Rf(z)=\lim_{r\rightarrow 0} \dfrac{f(z+rz)-f(z)}{r},$$
where $r$ is a real parameter and so $z+rz$ is a radial variation of $z$. Moreover, the $k$-th radial differentiation operator on $H(\mathbb{B})$ is denoted by $R^{(k)}$ and defined as usual.

Let $\beta: \mathbb{B} \times \mathbb{B} \rightarrow [0,\infty)$ be the Bergman metric,  $\varphi_z$ be an automorphism from $\mathbb{B}$ onto $\mathbb{B}$, satisfying $\varphi_z(0)=z$, $\varphi_z(z)=0$, and $\varphi_z\circ \varphi_z$ be the identity function. It is well-known that
\begin{equation*}
\beta(z,\zeta)= \dfrac{1}{2} \log \dfrac{1+|\varphi_z(\zeta)|}{1-|\varphi_z(\zeta)|}.
\end{equation*}
 If $r>0$ and $z\in \mathbb{B}$, the Bergman ball centered at $z$ with radius $r$ is defined as $B(z,r)=\{\zeta\in \mathbb{B}: \ \beta(z,\zeta)<r\}$.
For every $r>0$, the following inequalities hold:
\begin{equation} \label{e23}
 (1-|z|^2)^{n+1} \simeq (1-|\zeta|^2)^{n+1}\simeq V(B(z,r))\simeq V(B(\zeta,r)),
\end{equation}
 and
\begin{equation} \label{e24}
 |1-\langle z,\zeta\rangle|\simeq (1-|z|^2),
\end{equation}
for all $z\in \mathbb{B}$ and $\zeta\in B(z,r)$.

For $a\in \mathbb{B}$, let $P a$ be the radial projection of $a$ onto $\mathbb{S}$.
The pseudo–metric on $\mathbb{S}$ is defined as $\rho(z,\zeta) = |1-\langle z,\zeta\rangle|$.
As usual,
 $$D(z,r):=\{\zeta \in \mathbb{S} : \rho(z,\zeta) < r\}.$$
 For $z\in \mathbb{B}\setminus \{0\}$, let $Q_z=D(Pz,1-|z|)$. Then, the Carleson square $S(z)$ will be defined as
 $$S(z)=\{ \zeta\in \mathbb{B}, \ P\zeta\in Q_z, \ |z|<|\zeta|<1\},$$
and $S(0)=\mathbb{B}$. From \cite[Lemma 2]{du1}, if $\omega\in \widehat{\mathcal{D}}$, then
\begin{equation} \label{e15}
\widehat{\omega}(z) (1-|z|^2)^{n}\simeq V_\omega(S(z)).
\end{equation}

Let $\mu$ and $B$ be a Borel measure and a Borel subset of $\mathbb{B}$, respectively. Then, we set $\mu(B)=\int_B d\mu$.  For the function $g\in L^p(\mathbb{B},\mu)$, the average of $g$ based on $\mu$ is defined as
$$\langle g\rangle_{\mu,B}:= \dfrac{1}{\mu(B)} \int_B g(z) d\mu(z),$$
For convenience, $\langle h\rangle_{B}=\langle h\rangle_{V,B}$ for $h\in L^p(\mathbb{B},V)$.

\subsection{Weights}
If $\omega\in \mathcal{R}$, then for each $0\leq s<1$, there exists some constant $C=C_{s,\omega}>1$ such that
\begin{equation} \label{e16}
C^{-1} \omega(t)\leq \omega(r)\leq C \omega(t), \ \ 0\leq r\leq t\leq r+s(1-r)<1.
\end{equation}
For a proof, see \cite[Page 9 (i)]{pelaez1}. Equivalently, the regular weights are invariant in  the Bergman balls, that is, for some (all) $r>0$,
\begin{equation}  \label{e17}
\omega(z)\simeq \omega(\zeta), \qquad \zeta\in B(z,r).
\end{equation}
Indeed, since $B(z,r)$ is an ellipsoid of the (Euclidean) center $\frac{1-r^2}{1-r^2|z|^2}z$, its intersection with the line $\mathbb{C} z$ is a Euclidean disc of radius
\begin{equation*}
r \dfrac{1-|z|^2}{1-r^2|z|^2},
\end{equation*}
and its intersection with the affine subspace through z orthogonal to $\mathbb{C} z$ is a Euclidean ball of radius
\begin{equation*}
r \sqrt{\dfrac{1-|z|^2}{1-r^2|z|^2}},
\end{equation*}
one can see that (\ref{e16}) and (\ref{e17}) are equivalent. This and (\ref{e15}) imply that, when $\omega\in \mathcal{R}$, for all $r>0$, there holds
\begin{equation} \label{e19}
V_\omega(B(z,r))\simeq \omega(z)(1-|z|)^{n+1}\simeq \widehat{\omega}(z)(1-|z|)^n\simeq V_\omega(S(z)).
\end{equation}
Another known result is that for $\omega\in \widehat{\mathcal{D}}$, there exists some $t_0>0$ such that for all $t>t_0$ we have
\begin{equation} \label{e26}
\int_\mathbb{B} \dfrac{\omega(\zeta)}{|1-\overline{z}\zeta|^{t}} dV(\zeta)\simeq \dfrac{\widehat{\omega}(z)}{(1-|z|)^{t-1}}, \qquad z\in \mathbb{B}.
\end{equation}

\begin{proposition} \label{p3}
Let $0<p<\infty$ and $\omega\in \widehat{\mathcal{D}}$. Consider the weight $W(r)=W_\omega(r)=\widehat{\omega}(r)/(1-r)$. Then, the norms $\|.\|_{W,p}$ and $\|.\|_{\omega,p}$ are equivalent on $H(\mathbb{B})$ if and only if $\omega\in \mathcal{D}$.
Moreover, if $\omega\in \mathcal{D}$, then $W\in \mathcal{R}$.
\end{proposition}
\begin{proof}
The proof is similar to that of \cite[Proposition 5]{pelaez3}. Hence, we leave the details to the interested readers.
\end{proof}

\subsection{Dyadic structure on $\mathbb{B}$}
The discussions of this subsection come from \cite{arcozzi} and \cite{wick1}. Let
 $$S_r=\partial B(0,r)=\{w\in \mathbb{B}: \ \beta(0,w)=r\}=\Big{\{}w\in \mathbb{B}: \ \frac{1}{2} \log \frac{1+|w|}{1-|w|}=r\Big{\}}.$$
Fix $\theta,\lambda>0$. For $N\in \mathbb{N}$, we can find a
sequence of points $E_N = \{w_j^N\}_{j=1}^{J_N}$, and a corresponding sequence of Borel subsets $\{X_j^N\}^{J_N}_{j=1}$
of $S_{N\theta}$, that satisfy:
\begin{itemize}
\item[(i)] $S_{N\theta} = \bigcup^{J_N}_{j=1} X^N_j$,
\item[(ii)] $X^N_j\bigcap X^N_i=\emptyset, \ \ i\neq j$,
\item[(iii)] $S_{N\theta}\bigcap B(w_j^N, \lambda) \subset X^N_j \subset S_{N\theta}\bigcap B(w_j^N, C\lambda)$.
\end{itemize}

For $z\in \mathbb{B}$, let $P_r z$ denote the radial projection of $z$ onto the sphere $S_r$.
 We now define the set $K_j^N$ as follows:
\begin{eqnarray*}
K_1^0&:=&\{z\in\mathbb{B}: \ \beta(0,z)<\theta\},\\
K_j^N&:=&\{z\in\mathbb{B}: \ N\theta\leq \beta(0,z)<(N+1)\theta \ and \ P_{N\theta}z\in X_j^N\}, \ N\geq 1, \ j\geq 1.
\end{eqnarray*}
We will refer to the subset $K^N_j$ of $\mathbb{B}$ as a unit cube centered at $c^N_j=P_{(N+\frac{1}{2})\theta} w_j^N$,
while $K^0_1$ is centered at $0$. We say that $c^{N+1}_i$
is a child of $c^{N}_j$ if $P_{N\theta} c^{N+1}_i\in X^N_j$.

A tree structure is a set $\mathcal{T} := \{c^N_j\}$ which contains the centers of the cubes.
We will denote the elements of the tree by $\alpha$ and $\beta$, while $K_\alpha$ will be the cube with
center $\alpha$. There is no problem to abuse the notation and use $\alpha$ to denote both an element of
the tree $\mathcal{T}$ and the center of the corresponding cube or any element of the
cube. The notation $\beta\geq \alpha$ means that $\beta$ is a
descendant of $\alpha$, that is, there are $\alpha_1,...,\alpha_l\in \mathcal{T}$, such that $\beta$ is a child of $\alpha_l$, $\alpha_l$ is a child of $\alpha_{l-1}$, ..., and $\alpha_1$ is a child of $\alpha$. A $dyadic \ tent$ under $K_\alpha$ is defined as:
$$\widehat{K_\alpha}:= \bigcup_{\beta\in \mathcal{T}:\beta\geq \alpha} K_\beta.$$

\begin{lemma} \label{l7}
 The tree $\mathcal{T}$, constructed above with positive parameters $\lambda$ and $\theta$, satisfies the following properties.
\begin{itemize}
\item[(a)] $\mathbb{B} =\bigcup_{\alpha\in \mathcal{T}} K_\alpha$ and the sets $K_\alpha$ are pairwise disjoint. Furthermore, there are
constants $C_1$ and $C_2$ depending on $\lambda$ and $\theta$ such that for all $\alpha \in \mathcal{T}$ there
holds:
$$B(\alpha,C_1) \subset K_\alpha \subset B(\alpha,C_2);$$
\item[(b)] For any $R>0$, the balls $B(\alpha,R)$ satisfy the finite overlap condition
$$\sum_{\alpha\in \mathcal{T}} \mathbf{1}_{B(\alpha,R)} (z)\leq C_R, \qquad z\in \mathbb{B}.$$
\item[(c)]  $\widehat{K_\alpha}\simeq S(\alpha)$.
\item[(d)] If $\omega\in\widehat{\mathcal{D}}$, then $V_\omega(S(\alpha))\simeq V_\omega(\widehat{K_\alpha})$.
\end{itemize}
\end{lemma}
\begin{proof}
(a) and (b) have been proven in \cite{arcozzi}. From the definition of the Bergman metric and the tree structure, there is a $k\in \mathbb{N}$ such that for any $j$ and $N$ there is an $i$ such that
$$P_{N\theta} X_i^{N+k}\subseteq X_j^N$$
Since $k$ is independent of $i,j$ and $N$, we conclude (c). In addition, it is easy to see that (d) follows from  (a), (c), and the assumption that $\omega \in \widehat{\mathcal{D}}$.
\end{proof}

  A  $dyadic \ grid$ on $\mathbb{S}$ with calibre $\delta$ is a collection of Borel
subsets $\mathfrak{D} := \{Q^k_i\}_{i,k\in \mathbb{Z}}$ and points $\{z^k_i\}_{i,k\in \mathbb{Z}}$ in $\mathbb{S}$ that satisfy:
\begin{itemize}
\item[(i)] There are constants $c_1$ and $C_2$ such that for every $k, i \in \mathbb{Z}$ there holds:
$$D(z^k_i , c_1\delta^k ) \subset Q^k_i\subset D(z^k_i ,C_2\delta^k ).$$
\item[(ii)] For all $k\in\mathbb{Z}$, there holds $\mathbb{S} = \bigcup_{i\in\mathbb{Z}} Q^k_i$ and the sets are disjoint.
\item[(iii)] If $Q, R\in \mathfrak{D}$ and $Q \bigcap R\neq \emptyset$, then either $Q \subset R$ or $R \subset Q$.
\end{itemize}

 Consider $\mathfrak{D}$ as the dyadic grid on $\mathbb{S}$ with calibre $\delta$ as mentioned above. Then, from the proof of \cite[Lemma 3]{wick1}, this dyadic system creates a Bergman tree with suitable parameters $\theta$ and $\lambda$ which depend on $\delta$. Indeed, we can construct a tree structure as follows:
 $$X_j^N:=P_{N\theta} Q_j^N, \ \ c_j^N:= P_{(N+\frac{1}{2})\theta} z_j^N,$$
  there are $\lambda_2>\lambda_1>0$ so that
 $$S_{N\theta}\bigcap B(P_{N\theta} z_j^N, \lambda_1) \subset X^N_j \subset S_{N\theta}\bigcap B(P_{N\theta} z_j^N, \lambda_2)$$
Thus, if $\mathfrak{D} := \{Q^k_i\}_{i,k\in \mathbb{Z}}$ and $\mathcal{T}:= \{c^N_j\}_{j,N\in \mathbb{Z}}$ are constructed as above, then the function $c:\mathfrak{D}\rightarrow \mathcal{T}$, defined as $c(Q_j^N)=c_j^N$, is one-to-one and onto.

The following lemma is a special case of \cite[Theorem 4.1]{hytonen2} and also appears in \cite[Lemma 2]{wick1}.
\begin{lemma}\label{l8}
For every $\delta > 0$, there is an $M \in \mathbb{N}$ such that there is
a collection of dyadic systems of cubes $\{\mathfrak{D}^i\}_{i=1}^M$ with the following property: For every disc
$D(z, r )$, there is a $1\leq i \leq M$ such that there is a dyadic cube
$Q^k_j\in \mathfrak{D}^i$ with $D(z, r )\subset Q^k_j$ and $\delta^k \simeq r$ where the implied constants are independent of $z$
and $r$.
\end{lemma}

\begin{lemma}\label{l2}
Let $\omega\in \widehat{\mathcal{D}}$. For every $\delta > 0$, there is a finite collection of $dyadic$ $grids$ $\{\mathfrak{D}^i\}_{i=1}^M$ with the following property: For all $z\in \mathbb{B}$, there are $\zeta\in \mathbb{B}$ and $1\leq i \leq M$  such that there is a dyadic cube $Q \in\mathfrak{D}^i$ where  $S(z)\subseteq \widehat{K_{c(Q)}}\subseteq S(\zeta)$ and $V_\omega(\widehat{K_{c(Q)}})\simeq V_\omega(S(z))\simeq V_\omega(S(\zeta))$.
\end{lemma}
\begin{proof}
From Lemma \ref{l8}, there
is a $dyadic$ $grid$ $\mathfrak{D}^i$ from the finite collection and a $Q \in\mathfrak{D}^i$ such that the dyadic tent $\widehat{K_{c(Q)}}$
contains the Carleson square $S(z)$ and $\widehat{K_{c(Q)}}\simeq S(z)$, that is, there is a positive constant $C$ such that
$$ S(z)\subseteq \widehat{K_{c(Q)}}\subseteq C S(z).$$
One can see that there is a  $\zeta\in \mathbb{B}$ so that $CS(z)\subseteq S(\zeta)$ and $S(\zeta)\simeq S(z)$. Thus, the desired inclusions hold. Now, since $\omega\in \widehat{\mathcal{D}}$, we conclude that $V_\omega(\widehat{K_{c(Q)}})\simeq V_\omega(S(z))\simeq V_\omega(S(\zeta))$.
\end{proof}

\textbf{Notation.} From now on, for every calibre $\delta>0$, we assume that $\mathfrak{D}_\delta=\{\mathfrak{D}_\delta^i\}_{i=1}^M$ is the finite collection of $dyadic$ $grids$ obtained in Lemma \ref{l2}.

\subsection{Maximal operator}
For $\omega\in \widehat{\mathcal{D}}$ and $\varphi\in L^p(\mathbb{B},V_\omega)$, we define
\begin{align*}
\mathcal{M}_{\omega,\mathfrak{D}_\delta}(\varphi)(z)&=\sup_{z\in \widehat{K_{c(Q)}}, Q\in \mathfrak{D}_\delta} \dfrac{1}{V_\omega(\widehat{K_c(Q)})} \int_{\widehat{K_c(Q)}} |\varphi(\zeta)|dV_\omega(\zeta).
\end{align*}
and
\begin{align*}
\mathcal{M}_{\omega}(\varphi)(z)= \sup_{z\in S(a), a\in \mathbb{B}} \dfrac{1}{V_\omega(S(a))} \int_{S(a)} |\varphi(\zeta)|dV_\omega(\zeta).
\end{align*}
From Lemma \ref{l2}, $\mathcal{M}_{\omega}(\varphi)(z)\simeq \mathcal{M}_{\omega,\mathfrak{D}_\delta}(\varphi)(z)$.
 Therefore, \cite[Lemma 5]{du1} implies that for every $0<p<\infty$, there is some $C=C(p,\omega,n)>0$, such that
\begin{equation} \label{e21}
|f(z)|^p\leq C M_{\omega,\mathfrak{D}_\delta}(|f|^p)(z), \ \ f\in H(\mathbb{B}).
\end{equation}

Let $t\in \mathbb{R}$, $\mu$ be a positive Borel measure on $\mathbb{B}$, and  $\mathfrak{D}$ be a  $dyadic$ $grid$. We define
\begin{align*}
\mathcal{M}_{t,\mu,\mathfrak{D}}(\varphi)(z)&=\sup_{z\in \widehat{K_{c(Q)}}, Q\in \mathfrak{D}} \mu(\widehat{K_{c(Q)}})^t \langle \varphi\rangle_{\mu,\widehat{K_{c(Q)}}}.
\end{align*}
The proof of the following lemma is standard. Thus, we give its proof in the appendixes (subsection 7.1)
\begin{lemma}\label{l5}
Let $1<p\leq q<\infty$ and $0\leq t<1$, where $\frac{1}{p}-\frac{1}{q}=t$. Then $\mathcal{M}_{t,\mu,\mathfrak{D}}:L^p_\mu\rightarrow L^q_\mu$ is bounded.
\end{lemma}

\section{Sparse domination}
The inequality has been given in the following lemma has been given in Chapter 2 of \cite{zhu1}. Here, we want to check what relationship the implicit constant has with r and p.
\begin{lemma} \label{l4}
Let $0<r,p<\infty$. Then, there are positive constants $C_r$ and $D_r$ which depend only on $r$ such that
$$|\nabla f (z)|^p\leq C_r^p D_r \int_{B(0,3r)} |f(\zeta)|^p dV (\zeta) , \qquad z\in B(0,r),$$
for all $f\in H(\mathbb{B})$.
\end{lemma}
\begin{proof}
Let $f\in H(\mathbb{B}_n)$ and $0<s<1$. Then, $f(sz)$ is in the classical Bergman space $A^2$. Thus,
$$f(sz)=\int_\mathbb{B} \dfrac{ f(s\zeta)}{(1-\langle z,\zeta \rangle)^{n+1}} dV(\zeta)$$
By replacing $sz$ with $z$ and the change of variable $s\zeta\mapsto \zeta$, we obtain
$$f(z)=s^2 \int_{|\zeta|<s} \dfrac{ f(\zeta)}{(s^2-\langle z,\zeta \rangle)^{n+1}} dV(\zeta), \qquad |z|<s.$$
Thus, for $0<t<r$, there is a positive constant $C_{s,t}$ so that
$$\Big|\dfrac{\partial f}{\partial z^k} (z)\Big|\leq C_{s,t} \sup\{ |f(\zeta)|; \ |\zeta|<s\}, \qquad |z|\leq t.$$
This estimate shows that there is a positive constant $C_r$ so that
\begin{equation} \label{e31}
|\nabla f (z)|\leq C_r   \sup\{ |f(\zeta)|; \ \zeta\in B(0,2r)\}, \qquad z\in B(0,r),
\end{equation}
Furthermore, from the triangular inequality, we have $B(\zeta,r)\subset B(0,3r)$ for $\zeta\in B(0,2r)$. Hence, form the subharmonicity of $|f\circ \varphi_\zeta|^p$, we have
\begin{align} \label{e49}
|f(\zeta)|^p&=|f\circ \varphi_\zeta (0)|^p\leq \dfrac{1}{V(B(0,r)}\int_{B(0,r)} |f\circ \varphi_\zeta (a)|^p dV(a)\notag\\
&= \dfrac{1}{V(B(0,r)}\int_{B(\zeta,r)} |f (a)|^p \dfrac{(1-|\zeta|^2)^2}{|1-\langle a,\zeta\rangle|^4} dV(a)\\
&\leq D_r \int_{B(0,3r)} |f (a)|^p dV(a).\notag
\end{align}
The desired result is obtained by combining (\ref{e31}) and (\ref{e49}).
\end{proof}

Now, we give the main result of this section.

\begin{theorem} \label{t}
 Let  $0<p,\delta<\infty$ and $k \in \mathbb{N}\bigcup \{0\}$. Then, there are positive constants $r_\delta$, $C_\delta$, and $D_\delta$ which depend only on $\delta$ such that
 \begin{equation} \label{e12}
|R^{(k)} f(z)|^p \leq D_\delta^{k+1} C_\delta^{kp}  \sum_{i=1}^M \sum_{Q\in\mathfrak{D}^i_\delta}  \dfrac{\mathbf{1}_{K_{c(Q)}}(z)}{(1-|c(Q)|)^{kp}} \langle |f|^{p}\rangle_{B(c(Q),2^{k+1}r_\delta)},
\end{equation}
for all $f\in H(\mathbb{B})$ and $z\in \mathbb{B}$. Moreover, if $\omega\in \widehat{\mathcal{D}}$,  then for each  $Q\in\mathfrak{D}_\delta$, we have
\begin{equation} \label{e36}
\langle |f|^p\rangle_{B(c(Q),2^{k+1}r_\delta)} \leq C_\delta \mathcal{M}_{\omega,\mathcal{T}}(|f|^p)(z), \ \ f\in H(\mathbb{B}), \ z\in B(c(Q),2^{k+1}r_\delta).
\end{equation}
\end{theorem}
\begin{proof}
For every $1\leq i\leq M$, we know that $\mathbb{B}=\bigcup_{Q\in \mathfrak{D}_\delta^i} K_{c(Q)}$ and the sets $K_{c(Q)}$, $Q\in \mathfrak{D}_\delta^i$, are disjoint. This implies that
$$|R^{(k)} f(z)|^p = \dfrac{1}{M}  \sum_{i=1}^M \sum_{Q\in\mathfrak{D}^i_\delta}  \mathbf{1}_{K_{c(Q)}}(z)  |R^{(k)} f(z)|^p $$

Fix $z\in \mathbb{B}$ and $Q\in \mathfrak{D}_\delta$ such that $z\in K_{c(Q)}$. Let $r_\delta= \sup_{1\leq i\leq M} C_{2,i}$, where $C_{2,i}$ is as in Lemma \ref{l7} (i) for the tree structure induced by $\mathfrak{D}_\delta^i$.
 From the subharmonicity of $|f\circ \varphi_z|^p$, we have
\begin{align*}
|f(z)|^p &\leq \dfrac{1}{ V(B(0,r_\delta))} \int_{B(0,r_\delta)} |f\circ \varphi_z(\zeta)|^p dV(\zeta)\\
&= \dfrac{1}{ V(B(0,r_\delta))} \int_{B(z,r_\delta)} |f(\zeta)|^p |\varphi_z^\prime(\zeta)|^2 dV(\zeta)\\
&\leq \dfrac{L_\delta}{ V(B(z,r_\delta))} \int_{B(z,r_\delta)} |f(\zeta)|^p dV(\zeta)\\
&\leq \dfrac{D_\delta}{ V(B(c(Q),2r_\delta))} \int_{B(c(Q),2r_\delta)} |f(\zeta)|^p dV(\zeta)
\end{align*}
Thus, we obtain (\ref{e12}) for $k=0$.
 From \cite[Lemma 2.14]{zhu1},
$$|R f(z)|\leq \dfrac{1}{1-|z|^2} |\bigtriangledown (f \circ \varphi_z)(0)|.$$
Hence, by Lemma \ref{l4} and the change of variable $\zeta\mapsto \varphi_z(\zeta)$, we have
\begin{equation*}
|R f(z)|^p \leq \dfrac{D_\delta^2 C_\delta^p}{(1-|z|^2)^{p} V(B(z,3r_\delta))} \int_{B(z,3r_\delta)} |f(\zeta)|^p dV(\zeta).
\end{equation*}
Again, since $z\in K_{c(Q)}$, we obtain (\ref{e12}) for $k=1$. Now, (\ref{e12}) follows by induction on $k$.

Now, from (\ref{e21}),
\begin{equation} \label{e13}
\langle |f|^p\rangle_{B(c(Q),2^{k+1}r_\delta)} \leq C \sup_{\zeta\in B(c(Q),2^{k+1}r_\delta)} \mathcal{M}_{\omega,\mathcal{T}}(|f|^p)(\zeta).
\end{equation}
 For every $\zeta\in B(z,2^{k+1}r_\delta)$ and $a_\zeta$ where $\zeta\in S(a_\zeta)$ there is some $a_z$ where $z\in S(a_z)$, $S(a_\zeta)\subseteq S(a_z)$, and $S(a_\zeta)\simeq S(a_z)$. This, with (\ref{e13}) and Lemma \ref{l2}, implies (\ref{e36}).
\end{proof}

\textbf{Notation.} From now on, for $k\in \mathbb{N}\bigcup \{0\}$, we fix  $R_\delta=2^{k+1}r_\delta$ and $B_{c(Q)}=B(c(Q),R_\delta)$ for $Q$ in $\mathfrak{D}_\delta$.

\section{Forward Bergman Carleson measure}

In this section, we  characterize the positive Borel measures $\mu$ on $\mathbb{B}$ so that  $R^{(k)}:A^p_\omega\rightarrow L^q_\mu$ is bounded when $0< p\leq q<\infty$.

Let $\nu$ be a weight on $\mathbb{B}$, $k\in \mathbb{N}\bigcup \{0\}$ and $\delta,s$, and $t$ be positive numbers. We say that a positive Borel measure $\mu$ is in $\mathcal{C}_{\nu,\delta}^{t,ks}$ if
\begin{equation*}
[\mu]_{\nu,\delta}^{t,ks}:= \sup_{Q\in  \mathfrak{D}_\delta} \dfrac{\mu(B_{c(Q)})}{V_\nu(B_{c(Q)})^t (1-|c(Q)|^2)^{ks}}<\infty.
\end{equation*}
Now, we present the main result of this section.

\begin{theorem} \label{t1}
Let $\omega\in \mathcal{D}$, $\nu(r)=\widehat{\omega}(r)/(1-r)$, $0<p\leq q<\infty$, and $k\in \mathbb{N}\bigcup \{0\}$. If $\mu$ is a positive Borel measure,
then the following statements are equivalent.
\begin{itemize}
\item[(i)]  $R^{(k)}:A^p_\omega\rightarrow L^q_\mu$ is bounded,
\item[(ii)] $\mu\in \mathcal{C}_{\nu,\delta}^{\frac{q}{p},kq}$ for some (every) calibre $\delta>0$.
\item[(iii)] For some (every) calibre $\delta>0$ and   $m\in \mathbb{N}$, there are positive numbers $C_\delta$ and $D_\delta$ which depend only on $\delta$ so that
\begin{equation*}
\|R^{(k)} f\|_{\mu,q}^q \leq  D_\delta^{m(k+1)} C_\delta^{kq} [\mu]_{\nu,\delta}^{\frac{q}{p},kq} \inf_{r_j>0, \ r_1+...+r_m=q} \Big( \sum_{i=1}^M \sum_{Q\in\mathfrak{D}^i}  V_\nu(B_{c(Q)})^\frac{q}{p}. \prod_{j=1}^m \langle |f|^{r_j}\rangle_{B_{c(Q)}}\Big),
\end{equation*}
for all $f\in A_\omega^p$.
\end{itemize}
\end{theorem}
\begin{proof}
According to Proposition \ref{p3} and without loss of generality, we can assume $\omega(r)=\nu(r):=\widehat{\omega}(r)/(1-r)$ and hence $\omega\in \mathcal{R}$.

(i) $\Rightarrow$ (ii). For every $Q\in \mathfrak{D}_\delta$, it follows from (\ref{e24}), (i), (\ref{e26}), and (\ref{e19})  that there exists a $\gamma>0$ such that
\begin{align*}
\mu(B_{c(Q)})&\simeq (1-|c(Q)|^2)^{(k+\gamma)q} \int_{B_{c(Q)}} \Big|R^{(k)}\Big(\dfrac{1}{1-\overline{c(Q)}\zeta}\Big)^\gamma\Big|^q d\mu(\zeta)\notag\\
&\leq  (1-|c(Q)|^2)^{(k+\gamma)q} \int_{\mathbb{B}} \Big|R^{(k)}\Big(\dfrac{1}{1-\overline{c(Q)}\zeta}\Big)^\gamma\Big|^q d\mu(\zeta)\notag\\
&\lesssim  (1-|c(Q)|^2)^{(k+\gamma)q} \Big( \int_{\mathbb{B}} \dfrac{1}{|1-\overline{c(Q)}\zeta|^{p\gamma}} dV_\omega(\zeta)\Big)^\frac{q}{p}\notag\\
&\lesssim \dfrac{(1-|c(Q)|^2)^{(k+\gamma)q} \omega(c(Q))^\frac{q}{p}}{(1-|c(Q)|^2)^{(p\gamma-2)q/p}}\\
&\simeq V_\omega(B_{c(Q)})^\frac{q}{p} (1-|c(Q)|^2)^{kq}.
\end{align*}
Thus, $\mu$ is in $\mathcal{C}_{\nu,\delta}^{\frac{q}{p},kq}$.

(ii) $\Rightarrow$ (iii). Fix $\delta>0$ and $m\in \mathbb{N}$.
Let $f\in A_\omega^q$ and  $r_j>0$, where $r_1+...+r_m=q$.
Using  (\ref{e12}) $m$ times, there are positive numbers $C_\delta$ and $D_\delta$ such that
\begin{align*}
\|R^{(k)}  f \|_{q,\mu}^q&= \int_{\mathbb{B}}   |R^{(k)} f (\zeta)|^q  d\mu(\zeta)= \int_{\mathbb{B}}   |R^{(k)} f (\zeta)|^{q-r_1} |R^{(k)} f (\zeta)|^{r_1} d\mu(\zeta)\notag\\
&\leq D_\delta^{k+1} C_\delta^{kr_1} \int_{\mathbb{B}}  |R^{(k)}f(\zeta)|^{q-r_1}  \Big[ \sum_{i=1}^M \sum_{Q\in\mathfrak{D}^i_\delta} \dfrac{\mathbf{1}_{K_{c(Q)}}(\zeta)}{(1-|c(Q)|)^{kr_1}} \langle |f|^{r_1}\rangle_{B_{c(Q)}} \Big] d\mu(\zeta)\\
&=  D_\delta^{k+1} C_\delta^{kr_1} \sum_{i=1}^M \sum_{Q\in\mathfrak{D}^i_\delta} \dfrac{\langle |f|^{r_1}\rangle_{B_{c(Q)}}}{(1-|c(Q)|)^{kr_1}} \int_{K_{c(Q)}}  |R^{(k)}f(\zeta)|^{q-r_1} d\mu(\zeta)\\
&\leq D_\delta^{2(k+1)} C_\delta^{k(r_1+r_2)} \sum_{i=1}^M \sum_{Q\in\mathfrak{D}^i_\delta} \dfrac{\langle |f|^{r_1}\rangle_{B_{c(Q)}} \langle |f|^{r_2}\rangle_{B_{c(Q)}}}{(1-|c(Q)|)^{k(r_1+r_2)}} \int_{K_{c(Q)}}  |R^{(k)}f(\zeta)|^{q-(r_1+r_2)} d\mu(\zeta)\\
&\leq D_\delta^{m(k+1)} C_\delta^{kq}   \sum_{i=1}^M \sum_{Q\in\mathfrak{D}^i_\delta} \dfrac{\mu(K_{c(Q)})}{(1-|c(Q)|)^{kq}} \prod_{j=1}^m \langle |f|^{r_j}\rangle_{B_{c(Q)}} .
\end{align*}
Now, (iii) follows from (ii).

(iii) $\Rightarrow$ (i).
 Let $f\in A^p_\omega$, then there is some $i_0\in\{1,...,M\}$ such that
\begin{align} \label{e43}
\sum_{i=1}^M \sum_{Q\in\mathfrak{D}^i_\delta} V_\omega(B_{c(Q)})^\frac{q}{p}. \prod_{j=1}^m \langle |f|^{r_j}\rangle_{B_{c(Q)}}
&\leq M \sum_{Q\in\mathfrak{D}^{i_0}_\delta} V_\omega(B_{c(Q)})^\frac{q}{p} . \prod_{j=1}^m \langle |f|^{r_j}\rangle_{B_{c(Q)}}.
\end{align}
Take $\alpha>0$ such that
$$\alpha p>1, \qquad \frac{\frac{q}{p}-1}{\alpha}<1.$$
Then, from (iii), (\ref{e43}), (\ref{e36}), and Lemma \ref{l5}, we have
\begin{align}
\|R^{(k)}f\|_{\mu,q}^q &\lesssim \sum_{Q\in\mathfrak{D}^{i_0}_\delta} V_\omega(B_{c(Q)})^\frac{q}{p} . \prod_{j=1}^m \langle |f|^{r_j}\rangle_{B_{c(Q)}}\notag\\
 &\lesssim \sum_{Q\in\mathfrak{D}^{i_0}_\delta} V_\omega(B_{c(Q)}) .\prod_{j=1}^m  \mathcal{M}_{\frac{\frac{q}{p}-1}{\alpha r_j} ,V_\omega,\mathfrak{D}_\delta}(|f|^\frac{1}{\alpha})(c(Q))^{\alpha r_j}\notag\\
  &= \sum_{Q\in\mathfrak{D}^{i_0}_\delta} V_\omega(B_{c(Q)})  \mathcal{M}_{\frac{\frac{q}{p}-1}{\alpha} ,V_\omega,\mathfrak{D}_\delta}(|f|^\frac{1}{\alpha})(c(Q))^{\alpha}. \label{e14}
\end{align}
Since $\{B_{c(Q)}\}_{Q\in \mathfrak{D}^{i_0}_\delta}$ has a finite overlap (\ref{e36}) implies that
\begin{align*}
(\ref{e14})&\lesssim \int_\mathbb{B}  \mathcal{M}_{\frac{\frac{q}{p}-1}{\alpha},V_\omega,\mathfrak{D}_\delta}(|f|^\frac{1}{\alpha})(z)^{\alpha q} dV_\omega(z) \\
&\lesssim \Big( \int_\mathbb{B} |f(z)|^{\frac{1}{\alpha}.\alpha p}  dV_\omega(z)  \Big)^\frac{\alpha q}{\alpha p}=\|f\|_{p,\omega}^q.
\end{align*}
The proof is complete.
\end{proof}

\begin{remark} \label{r1}
Let $\omega\in \widehat{\mathcal{D}}$, $t\in \mathbb{R}$, and $\delta>0$. Define
$$M_{t,\omega,\mathfrak{D}_\delta}(\varphi)(z):= (1-|z|^2)^t \mathcal{M}_{\omega,\mathfrak{D}_\delta}(\varphi)(z), \ z\in \mathbb{B}.$$
and
$$M_{t,\omega}(\varphi)(z):= (1-|z|^2)^t \mathcal{M}_{\omega}(\varphi)(z), \ z\in \mathbb{B}.$$
By Theorem \ref{t}, for every $s>0$, we have
$$|R^{(k)}f(z)|^s\lesssim M_{-ks,\omega,\mathfrak{D}_\delta}(|f|^s)(z).$$
 Thus, similar the proof of \cite[Theorem 2.1]{pelaez1}, we can obtain the following result. We just need to substitute $M_{-kq,\omega}$ with  $\mathcal{M}_{\omega}$ in its proof.
The equivalence of (\ref{e11}) and (\ref{e32}) comes from Lemma \ref{l2}.

\begin{itemize}
\item Let $\omega\in \widehat{\mathcal{D}}$, $0< p\leq q<\infty$, and $\mu$ be a positive Borel measure on $\mathbb{B}$. Then,
 $R^{(k)}:A^p_\omega\rightarrow L^q_\mu$ is bounded if and only if for some (every) calibre $\delta>0$,
\begin{equation}\label{e11}
\sup_{Q\in \mathfrak{D}_\delta} \dfrac{\mu(\widehat{K_{c(Q)}})}{V_\omega(\widehat{K_{c(Q)}})^{\frac{q}{p}}(1-|c(Q)|^2)^{kq}}<\infty,
\end{equation}
if and only if
\begin{equation} \label{e32}
[\mu]_{\omega,\frac{q}{p}}:=\sup_{a\in \mathbb{B}} \dfrac{\mu(S(a))}{V_\omega(S(a))^{\frac{q}{p}}(1-|a|^2)^{kq}}<\infty.
\end{equation}
\end{itemize}
\end{remark}

\begin{corollary} \label{c1}
Let $\mu$ be a positive Borel measure on $\mathbb{B}$, $\omega\in \widehat{\mathcal{D}}$, $0<p
\leq q<\infty$, and $\alpha>0$ be such that $\alpha q>1$.
Then,  $R^{(k)}:A^p_\omega\rightarrow L^q_\mu$ is bounded  if and only if for some (every) calibre $\delta>0$, $M_{-\frac{kq}{\alpha},\omega,\mathfrak{D}_\delta}((.)^\frac{1}{\alpha})^\alpha: L_\omega^p\rightarrow L^q_\mu$ is bounded. Moreover, for every $\delta>0$,
$$\|M_{-\frac{kq}{\alpha},\omega,\mathfrak{D}_\delta}((.)^\frac{1}{\alpha})^\alpha: L_\omega^p\rightarrow L^q_\mu\| \simeq \sup_{Q\in \mathfrak{D}_\delta} \dfrac{\mu(\widehat{K_{c(Q)}})}{V_\omega(\widehat{K_{c(Q)}})^{\frac{q}{p}}(1-|c(Q)|^2)^{kq}}.$$
\end{corollary}

\section{Vanishing Bergman Carleson measure}

Now, we address the compactness of $R^{(k)}$. The following lemma eliminates the need for a weak topology in the study of compact operators from $A^p_\omega$ to $L_\mu^q$.

\begin{lemma}\cite[Lemma 4]{du1} \label{l10}
Let $0<p,q<\infty$, $\omega\in \widehat{\mathcal{D}}$, and $\mu$ be a positive Borel measure on $\mathbb{B}$. If $T:A^p_\omega\rightarrow L_\mu^q$ is linear and bounded, then $T$ is compact if and only if whenever $\{f_k\}$ is bounded in $A^p_\omega$ and $f_k\rightarrow 0$ uniformly on the compact subsets of $\mathbb{B}$, $\lim_{k\rightarrow \infty} \|T f_k\|_{L_\mu^q}=0$.
\end{lemma}

Now, we define a subclass of $\mathcal{C}_{\nu,\delta}^{t,ks}$. The class of measures $\mathcal{V}_{\nu,\delta}^{t,ks}$ contains the positive Borel measures $\mu$ on $\mathbb{B}$ such that
\begin{equation*}
\lim_{Q\in  \mathfrak{D}_\delta, |c(Q)|\rightarrow 1} \dfrac{\mu(B_{c(Q)})}{V_\nu(B_{c(Q)})^\frac{q}{p} (1-|c(Q)|^2)^{kq}}=0.
\end{equation*}

\begin{theorem} \label{t2}
Let $\omega\in \mathcal{D}$, $\nu(r)=\widehat{\omega}(r)/(1-r)$, and $0<p\leq q<\infty$. If $\mu$ is a positive Borel measure such that $R^{(k)}:A^p_\omega\rightarrow L^q_\mu$ is bounded,
then the following statements are equivalent.
\begin{itemize}
\item[(i)]  $R^{(k)}:A^p_\omega\rightarrow L^q_\mu$ is compact.
\item[(ii)] $\mu \in \mathcal{V}_{\nu,\delta}^{t,s}$, for some (every) calibre $\delta>0$.
\item[(iii)]  Let $m\in \mathbb{N}$, $r_1+...+r_m=q$, where $0<r_j\leq q$. Then, for any bounded set $\{f_l\}_{l\geq 1} \subset A^p_\omega$ with $f_l \rightarrow 0$, uniformly on the compact subsets of $\mathbb{B}$, as $l \rightarrow \infty$,
\begin{align*}
&\lim_{s\rightarrow 1^-} \sup_{l\geq 1}  \Big( \sum_{i=1}^M \sum_{Q\in\mathfrak{D}^i_\delta, |c(Q)|>s}  \dfrac{\mu(K_{c(Q)})}{(1-|c(Q)|^2)^{kq}} .  \prod_{j=1}^m \langle |f_l|^{r_j}\rangle_{B_{c(Q)}} \Big)=0.
\end{align*}
\end{itemize}
\end{theorem}
\begin{proof}
As in the proof of boundedness, without loss of generality, let $\omega=\nu\in \mathcal{R}$.

(i) $\Rightarrow$ (ii). If
$$f_{z,\omega}(\zeta)=\Big(\dfrac{1}{1-\overline{z}\zeta}\Big)^t,$$
where $t$ is as in (\ref{e26}), then $g_{z,\omega}=f_{z,\omega}/\|f_{z,\omega}\|_{p,\omega}$ is a bounded sequence in $A_\omega^p$, where $g_{z,\omega}$ converges to the zero function, uniformly on the compact subsets of $\mathbb{B}$, as $|z|\rightarrow 1^{-}$.
Since $R^{(k)}:A^q_\omega\rightarrow L^q_\mu$ is compact,  Lemma \ref{l10}, implies that $\|R^{(k)} g_{z,\omega}\|_{q,\mu}\rightarrow 0$,   as $|z|\rightarrow 1^{-}$. Therefore,
\begin{align*}
\dfrac{\mu(B_{c(Q)})}{V_\omega(B_{c(Q)})^\frac{q}{p} (1-|c(Q)|^2)^{kq}}&\simeq  \int_{B_{c(Q)}} \Big|R^{(k)} g_{c(Q),\omega}(\zeta)\Big|^q d\mu(\zeta)\notag\\
&\leq \int_{\mathbb{B}} \Big|R^{(k)} g_{c(Q),\omega}(\zeta)\Big|^q d\mu(\zeta) \rightarrow 0,\notag
\end{align*}
as $|c(Q)|\rightarrow 1^{-}$. Thus, $\mu \in \mathcal{V}_{\nu,\delta}^{t,s}$.

(ii) $\Rightarrow$ (iii). Fix $\varepsilon>0$.
Since $\mu \in \mathcal{V}_{\nu,\delta}^{t,s}$, there is some $0<s<1$ such that  for $Q\in \mathfrak{D}_\delta$ with $|c(Q)|>s$, we have
$$\dfrac{\mu(B_{c(Q)})}{V_\omega(B_{c(Q)})^\frac{q}{p}(1-|c(Q)|^2)^{kq}}<\varepsilon,$$
Thus,
\begin{align}
&\sum_{Q\in \mathfrak{D}^i_\delta, \ |c(Q)|>s}\dfrac{\mu(B_{c(Q)})}{(1-|c(Q)|^2)^{kq}} . \prod_{j=1}^m \langle |f_l|^{r_j}\rangle_{B_{c(Q)}}\notag\\
& \qquad \qquad \qquad \qquad \lesssim \varepsilon .\sum_{Q\in \mathfrak{D}^i_\delta, \ |c(Q)|>s} V_\omega(B_{c(Q)})^\frac{q}{p}  . \prod_{j=1}^m \langle |f_l|^{r_j}\rangle_{B_{c(Q)}}\notag\\
& \qquad \qquad \qquad \qquad \lesssim\varepsilon \|f_l\|_{p,\omega}^q.\notag
\end{align}
Here, in the last inequality, we use the proof of boundedness. Since $\varepsilon>0$ was arbitrary, the proof is complete.

(iii) $\Rightarrow$ (i).
Let $\{f_l\}_{l\in\mathbb{N}}\subset A^p_\omega$ be a bounded set, satisfying $f_{l}\rightarrow 0$ as $l\rightarrow \infty$, uniformly on the compact subsets of $\mathbb{B}$.
Let $\varepsilon>0$ and $m\in \mathbb{N}$. By (\ref{e12}), for each $l,N \geq 1$, we have
\begin{align}
\|R^{(k)} f_l\|_{q,\mu}^q &\lesssim \sum_{i=1}^M \sum_{Q\in\mathfrak{D}^i_\delta}  \dfrac{\mu(B_{c(Q)})}{(1-|c(Q)|^2)^{kq}}
 . \prod_{j=1}^m \langle |f_l|^{r_j}\rangle_{B_{c(Q)}} \notag\\
 &\lesssim \sum_{i=1}^M \sum_{Q\in\mathfrak{D}^i_\delta, |c(Q)|\leq s}   \dfrac{\mu(B_{c(Q)})}{(1-|c(Q)|^2)^{kq}}
   . \prod_{j=1}^m \langle |f_l|^{r_j}\rangle_{B_{c(Q)}}\label{e28} \\
 & \ \ \ +\sum_{i=1}^M \sum_{Q\in\mathfrak{D}^i_\delta, |c(Q)|>s}   \dfrac{\mu(B_{c(Q)})}{(1-|c(Q)|^2)^{kq}}
   . \prod_{j=1}^m \langle |f_l|^{r_j}\rangle_{B_{c(Q)}} \label{e27},
\end{align}
where $s\in (0,1)$. Let $B_s$ be the closure of
$$\bigcup_{Q\in\mathfrak{D}_\delta, |c(Q)|\leq s} B_{c(Q)}.$$
It is clear that $B_s$ is a compact set. Thus, for any $s\in (0,1)$,
\begin{align}
(\ref{e28}) &\lesssim \sum_{Q\in\mathfrak{D}^i_\delta, |c(Q)|\leq s}  V_\omega(B_{c(Q)})^\frac{q}{p}
  . \prod_{j=1}^m \langle |f_l|^{r_j}\rangle_{B_{c(Q)}} \notag\\
&\lesssim \sum_{Q\in\mathfrak{D}^i_\delta, |c(Q)|\leq s} V_\omega(B_{c(Q)})^\frac{q}{p}  . \prod_{j=1}^m  \mathcal{M}_{\omega,\mathfrak{D}_\delta}(|f_l|^\frac{1}{\alpha})(c(Q))^{\alpha r_j} \notag\\
&= \sum_{Q\in\mathfrak{D}^i_\delta, |c(Q)|\leq s} V_\omega(B_{c(Q)}) \mathcal{M}_{\frac{q}{p}-1,\omega,\mathfrak{D}_\delta}(|f_l|^\frac{1}{\alpha})(c(Q))^{\alpha q}\notag\\
&\lesssim \int_{B_s}\mathcal{M}_{\frac{q}{p}-1,\omega,\mathfrak{D}_\delta}(|f_l|^\frac{1}{\alpha})(z)^{\alpha q} dV_\omega\notag\\
&\lesssim \int_{B_s} |f_l|^{p} dV_\omega,\notag
\end{align}
where the last inequality follows from the fact that  the measure $\mathbf{1}_{B_s}dV_\omega$ is a doubling measure.
 This implies that for any $s\in (0,1)$, we can take $l$ large enough such that
\begin{equation} \label{e35}
\sum_{i=1}^M \sum_{Q\in\mathfrak{D}^i_\delta, |c(Q)|\leq s}   \dfrac{\mu(K_{c(Q)})}{(1-|c(Q)|^2)^{kq}}
   . \prod_{j=1}^m \langle |f_l|^{r_j}\rangle_{B_{c(Q)}} <\dfrac{\varepsilon}{2}.
\end{equation}
Fix such $l$. Then, by the assumption, there exists an $s_0\in (0,1)$ such that for any $s_0<s<1$,
\begin{equation} \label{e40}
\sum_{i=1}^M  \sum_{Q\in\mathfrak{D}^i_\delta, |c(Q)|> s} \dfrac{\mu(K_{c(Q)})}{(1-|c(Q)|^2)^{kq}}  . \prod_{j=1}^m \langle |f_l|^{r_j}\rangle_{B_{c(Q)}} <\dfrac{\varepsilon}{2}.
\end{equation}
The desired result then follows from (\ref{e35}) and (\ref{e40}).
\end{proof}

\begin{remark}
As we said before, Theorem \ref{t2} is an extension of \cite[Theorem 4.6]{wick2}. However, Condition (iii) in Theorem \ref{t2} is simpler than Condition (ii) in \cite[Theorem 4.6]{wick2}. The reason is that in this new version, the powers $\frac{1}{\gamma}$ and $\frac{1}{\gamma^\prime}$ and particularly the multiple $V_\omega(B_{c(Q)})^\frac{1}{\gamma}$ have been omitted.
\end{remark}

\begin{remark}
Again the proof of the following result is similar to that in \cite[Theorem 2.1]{pelaez1}. Hence, we omitted its proof.
\begin{itemize}
\item Let $\omega\in \widehat{\mathcal{D}}$, $0< p\leq q<\infty$, and $\mu$ be a positive Borel measure on $\mathbb{B}$ so that  $R^{(k)}:A^p_\omega\rightarrow L^q_\mu$ is bounded. Then,
 $R^{(k)}:A^p_\omega\rightarrow L^q_\mu$ is compact if and only if, for some (every) calibre $\delta>0$,
$$\lim_{Q\in \mathfrak{D}_\delta, |c(Q)|\rightarrow 1^-} \dfrac{\mu(\widehat{K_{c(Q)}})}{V_\omega(\widehat{ K_{c(Q)}})^\frac{q}{p}(1-|c(Q)|^2)^{kq}}=0.$$
\end{itemize}
\end{remark}

\section{Reverse Bergman Carleson measures}

In this section, the reverse Bergman Carleson measures will be characterized.
First, we give two results from \cite{luecking5} which are the best obtained results on this topic.\\

\textbf{Theorem A}. \cite[Theorem 4.2]{luecking5}
Let $\alpha>-1$,  $0<\varepsilon, C,q<\infty$, and $\mu$ be a $(q,A_\alpha^q)$-Carleson measure. Then, there is an $r>0$ such that if $G=\{z\in \mathbb{D}; \ \frac{\mu(B(z,r))}{(1-|z|)^{\alpha+2}}>\varepsilon\}$ is a $C$-dominated set for $A^q_\alpha(\mathbb{D})$, then $\mu$ is a reverse $(q,A_\alpha^q)$-Carleson measure.\\

Let $C$ be a positive constant and $G$ be a Borel subset of $\mathbb{B}$. Then, we say that $G$ is a $C$-dominated set for $A_\omega^q$, if
$$\int_G |f|^q dV_\omega \geq C \int_\mathbb{B} |f|^q dV_\omega, \ \ f\in A_\omega^q.$$
 Sometimes when the quantity of $C$ is not important for us, we shortly say that $G$ is a dominated set for $A_\omega^q$.\\

\textbf{Theorem B}.  \cite[Theorem 4.3]{luecking5}
Let $\alpha>-1$ and $0<q<\infty$. If $\mu$ is a reverse $(q,A_\alpha^q)$-Carleson measure,
then there are $r,\varepsilon>0$ such that $\frac{\mu(B(z,r))}{(1-|z|)^{\alpha+2}}>\varepsilon$ for all $z\in \mathbb{D}$.\\

Consider $\delta>0$. Define the sets
$$H_\varepsilon(\mu;\omega;\delta):= \{ Q\in\mathfrak{D}_\delta; \ \dfrac{\mu(B_{c(Q)})}{V_\omega(B_{c(Q)})}>\varepsilon\}, \ \ G_\varepsilon(\mu;\omega;\delta):=\bigcup_{Q\in H_\varepsilon(\mu;\omega;\delta)} B_{c(Q)},$$
and for $f\in A^q_\omega$,
$$H_\varepsilon(f;\mu;\delta;r):= \Big{\{} Q\in\mathfrak{D}_\delta; \ \langle |f|^{r} \rangle_{\mu, B_{c(Q)}}>\varepsilon \langle |f|^{r} \rangle_{B_{c(Q)}}\Big{\}}.$$

Now we present the main result of the paper.

\begin{theorem} \label{t5}
Let $\omega\in \mathcal{D}$, $\nu(r)=\widehat{\omega}(r)/(1-r)$, and $0<q<\infty$. If $\mu$ is $\omega$-Bergman Carleson measure,
then the following statements are equivalent.
\begin{itemize}
\item[(i)]  $\mu$ is a reverse $\omega$-Bergman Carleson measure.
\item[(ii)] $\mathcal{M}_{\mu,\mathfrak{D}_\delta}(|.|^\frac{1}{\alpha})^\alpha: A_\omega^q\rightarrow L_\mu^q$ is bounded below, for some (every) $\alpha,\delta>0$, where $\alpha q>1$.
\item[(iii)] For some (every) $m\in \mathbb{N}$ and calibre $\delta$, there is an $\varepsilon>0$ such that
\begin{equation*}
\| f\|_{\nu,q}^q \lesssim \inf_{r_1+...+r_m=q} \Big( \sum_{Q\in H_\varepsilon(\mu;
\nu;\delta)} V_\nu(B_{c(Q)}).  \langle |f|^{r_m}\rangle_{\mu,B_{c(Q)}} . \prod_{j=1}^{m-1} \langle |f|^{r_j}\rangle_{B_{c(Q)}}\Big), \ \ f\in A_\omega^q.
\end{equation*}
\item[(iv)] For some (every) $m\in \mathbb{N}$ and calibre $\delta$, there is an $\varepsilon>0$ such that
\begin{equation*}
\| f\|_{\nu,q}^q \lesssim \inf_{r_1+...+r_m=q} \Big( \sum_{Q\in H_\varepsilon(\mu;
\nu;\delta)} V_\omega(B_{c(Q)}) . \prod_{j=1}^m \langle |f|^{r_j}\rangle_{B_{c(Q)}}. \mathbf{1}_{H_\varepsilon(f;\mu;\delta;r_m)} (Q)\Big), \ \ f\in A_\omega^q.
\end{equation*}
\end{itemize}
\end{theorem}
\begin{proof}
Again, without loss of generality, let $\omega=\nu\in \mathcal{R}$.

(iv)$\Rightarrow$(iii) is obvious.

(iii)$\Rightarrow$(ii). Let (iii) hold for $m\in \mathbb{N}$ and $\alpha>0$ be such that $\alpha q>1$. Let $r_1,...,r_{m-1}$ satisfy $r_1+...+r_{m-1}+\frac{1}{\alpha}=q$. Then, from (iii), we have
\begin{align} \label{e33}
\| f\|_{\omega,q}^q &\lesssim \sum_{Q\in H_\varepsilon(\mu;\omega;\delta)}  V_\omega(B_{c(Q)}). \prod_{j=1}^{m-1} \langle |f|^{r_j}\rangle_{B_{c(Q)}}. \langle |f|^\frac{1}{\alpha}\rangle_{\mu, B_{c(Q)}}\notag\\
&= \sum_{Q\in H_\varepsilon(\mu;\omega;\delta)}  V_\omega(B_{c(Q)}). \prod_{j=1}^{m-1} \langle |f|^{r_j}\rangle_{B_{c(Q)}}.\dfrac{1}{\mu(B_{c(Q)})} \int_{B_{c(Q)}}  |f(\zeta)|^{\frac{1}{\alpha}} d\mu(\zeta)\notag\\
&\lesssim  \sum_{Q\in H_\varepsilon(\mu;\omega;\delta)}  \mu(B_{c(Q)}). \prod_{j=1}^{m-1} \langle |f|^{r_j}\rangle_{B_{c(Q)}}.\dfrac{1}{V_\omega(B_{c(Q)})} \int_{B_{c(Q)}}  |f(\zeta)|^{\frac{1}{\alpha}} d\mu(\zeta)\notag\\
&\lesssim  \sum_{Q\in H_\varepsilon(\mu;\omega;\delta)} \mu(B_{c(Q)}). \prod_{j=1}^{m-1} \langle |f|^{r_j}\rangle_{B_{c(Q)}}. \dfrac{1}{V_\omega(\widehat{ B_{c(Q)})}} \int_{\widehat{ B_{c(Q)}}}  |f(\zeta)|^{\frac{1}{\alpha}} d\mu(\zeta),
\end{align}
where $\widehat{ B_\alpha}=\bigcup_{\beta\geq \alpha} B_\beta$. If $z\in B_{c(Q)}$, then from Lemma \ref{l2}, there is a $Q^\prime \in \mathfrak{D}_\delta$, where $z\in \widehat{K_{c(Q^\prime)}}$, $\widehat{ B_{c(Q)}}\subseteq \widehat{ K_{c(Q^\prime)}}$, and $\widehat{ B_{c(Q)}}\simeq \widehat{ K_{c(Q^\prime)}}$. Since $\mu$ is a $\omega$-Bergman Carleson measure, it follows from Remark \ref{r1} that
$$\mu(\widehat{ K_{c(Q^\prime)})}\lesssim V_\omega(\widehat{ K_{c(Q^\prime)})} \simeq V_\omega(\widehat{ B_{c(Q)})}.$$
 Hence,
\begin{equation}\label{e22}
\dfrac{1}{V_\omega(\widehat{ B_{c(Q)})}} \int_{\widehat{ B_{c(Q)}}}  |f(\zeta)|^{\frac{1}{\alpha}} d\mu(\zeta) \lesssim \mathcal{M}_{\mu,\mathfrak{D}_\delta} (|f|^{\frac{1}{\alpha}})(z).
\end{equation}
 Thus, by (\ref{e22}), (\ref{e36}), H\"{o}lder's inequality, and Corollary \ref{c1} (for $k=0$), we obtain
\begin{align*}
(\ref{e33})
&\lesssim \int_\mathbb{B} \prod_{j=1}^{m-1} \mathcal{M}_{\omega,\mathfrak{D}_\delta} (|f|^{r_j})(z)  \mathcal{M}_{\mu,\mathfrak{D}_\delta} (|f|^{\frac{1}{\alpha}})(z) d\mu(z) \\
&\leq \prod_{j=1}^{m-1} \Big( \int_\mathbb{B} \mathcal{M}_{\omega,\mathfrak{D}_\delta} (|f|^{r_j})^\frac{q}{r_j} d\mu(z) \Big)^\frac{r_j}{q}  \Big( \int_\mathbb{B} \mathcal{M}_{\mu, \mathfrak{D}_\delta} (|f|^{\frac{1}{\alpha}})^{\alpha q} d\mu(z) \Big)^\frac{1}{\alpha q} \\
&\lesssim \|f\|_{q,\omega}^{q-\frac{1}{\alpha}} \|\mathcal{M}_{\mu, \mathfrak{D}_\delta} (|f|^{\frac{1}{\alpha}})^{\alpha}\|_{q,\mu}^{\frac{1}{\alpha}}.
\end{align*}
That is, $\|f\|_{\omega,q}\lesssim \|\mathcal{M}_{\mu, \mathfrak{D}_\delta} (|f|^{\frac{1}{\alpha}})^{\alpha}\|_{q,\mu}$ which is the desired result.

(ii)$\Rightarrow$(i). From the boundedness of $\mathcal{M}_{\mu,\mathfrak{D}_\delta}:L^{\alpha q}(\mu)\rightarrow L^{\alpha q}(\mu)$ and the boundedness below of $\mathcal{M}_{\mu,\mathcal{T}}(|.|^\frac{1}{\alpha})^\alpha: A_\omega^q\rightarrow L_\mu^q$, we obtain
$$\|f\|_{\omega,q}\lesssim \|\mathcal{M}_{\mu,\mathcal{T}}(|f|^\frac{1}{\alpha})^\alpha\|_{\mu,q}\lesssim \|f\|_{\mu,q}.$$

(i)$\Rightarrow$(iv). Let (i) hold but (iv) not hold. Thus, there are an $m\in \mathbb{N}$ such that for every $k\in \mathbb{N}$, there are a sequence $0<r^k_1,..., r^k_m \leq q$, where $r^k_1+...+r^k_m=q$, and $f_k\in A_\omega^q$, where $\|f_k\|_{\omega,q}=1$ such that
\begin{equation} \label{e34}
\lim_{k\rightarrow \infty}\sum_{Q\in H_\frac{1}{k}(\mu;\omega;\delta)} V_\omega(B_{c(Q)}) . \prod_{j=1}^m \langle |f|^{r_j^k}\rangle_{B_{c(Q)}} . \mathbf{1}_{H_\frac{1}{k}(f_k;\mu;r_m^k)} (Q)=0.
\end{equation}
By using (\ref{e12}), one can see that
\begin{align*}
\|f_k \|_{q,\mu}^q &\lesssim \sum_{i=1}^M \sum_{Q\in\mathfrak{D}^i_\delta} . \prod_{j=1}^{m-1} \langle |f|^{r_j^k}\rangle_{B_{c(Q)}} . \int_{K_{c(Q)}} |f_k|^{r_m^k}d\mu \\
&=  \sum_{Q\in H_\frac{1}{k}(\mu;\omega;\delta)} \prod_{j=1}^{m-1} \langle |f|^{r_j^k}\rangle_{B_{c(Q)}} .\int_{K_{c(Q)}} |f_k|^{r_m^k}d\mu. \mathbf{1}_{H_\frac{1}{k}(f_k;\mu;r_k)} (c(Q))\\
& \ \ \ +  \sum_{Q\in \mathfrak{D}_\delta \setminus H_\frac{1}{k}(\mu;\omega;\delta)} \prod_{j=1}^{m-1} \langle |f|^{r_j^k}\rangle_{B_{c(Q)}} .\int_{K_{c(Q)}} |f_k|^{r_m^k}d\mu\\
&\ \ \ +  \sum_{Q\in  \mathfrak{D}_\delta \setminus H_\frac{1}{k}(f_k;\mu;r_k)}  \prod_{j=1}^{m-1} \langle |f|^{r_j^k}\rangle_{B_{c(Q)}} \int_{K_{c(Q)}} |f_k|^{r_m^k}d\mu\\
&=I_1^k+I_2^k+I_3^k.
\end{align*}

Again from (\ref{e12}) and the assumption that $\mu$ is a $\omega$-Bergman Carleson measure, we obtain
\begin{align*}
I_1^k&\lesssim \sum_{Q\in H_\frac{1}{k}(\mu;\omega;\delta)} V_\omega(B_{c(Q)}) . \prod_{j=1}^m \langle |f|^{r_j^k}\rangle_{B_{c(Q)}} . \mathbf{1}_{H_\frac{1}{k}(f_k;\mu;r_m^k)} (c(Q)).
\end{align*}
Thus, (\ref{e34}) implies that $\lim_{k\rightarrow\infty} I_1^k=0$.

 If $Q\in \mathfrak{D}_\delta \setminus H_\frac{1}{k}(\mu;\omega;\delta)$, then $\mu(B_{c(Q)})\leq \frac{1}{k} V_\omega(B_{c(Q)})$. Thus, from (\ref{e12}),
\begin{align*}
I_2^k &\lesssim \sum_{Q\in \mathfrak{D}_\delta \setminus H_\frac{1}{k}(\mu;\omega;\delta)} \mu(B_{c(Q)}). \prod_{j=1}^m \langle |f|^{r_j^k}\rangle_{B_{c(Q)}}\\
&\leq \dfrac{1}{k} \sum_{Q\in \mathfrak{D}_\delta \setminus H_\frac{1}{k}(\mu;\omega;\delta)} V_\omega(B_{c(Q)}). \prod_{j=1}^m \langle |f|^{r_j^k}\rangle_{B_{c(Q)}}\\
&\lesssim \dfrac{1}{k} \|f_k\|_{\omega,q}^q=\dfrac{1}{k}\rightarrow 0,
\end{align*}
as $k\rightarrow \infty$. For the third quantity, from the assumption we have
\begin{align*}
I_3^k &= \sum_{Q\in  \mathfrak{D}_\delta \setminus H_\frac{1}{k}(f_k;\mu;r_k)} \mu(B_{c(Q)}). \prod_{j=1}^{m-1} \langle |f|^{r_j^k}\rangle_{B_{c(Q)}}. \langle |f|^{r_m^k}\rangle_{\mu,B_{c(Q)}}\\
 &\leq \dfrac{1}{k} \sum_{Q\in  \mathfrak{D}_\delta \setminus H_\frac{1}{k}(f_k;\mu;r_k)} V_\omega(B_{c(Q)}). \prod_{j=1}^{m} \langle |f|^{r_j}\rangle_{B_{c(Q)}}\\
& \lesssim \dfrac{1}{k} \|f_k\|_{\omega,q}^q=\dfrac{1}{k}\rightarrow 0,
\end{align*}
as $k\rightarrow \infty$. Thus, $\lim_{l\rightarrow\infty}\|f_l\|_{q,\mu}=0$ which contradicts (i). Therefore, (iv) must hold.
\end{proof}

\begin{remark}
Note that we can choose the power $r_m$ in Theorem \ref{t5} (iii, iv) as small as we like.
\end{remark}

\begin{remark}
The gap between Theorems $A$ and $B$ is that the $r$ obtained in Theorem $A$ is small, while the $r$ in Theorem $B$ is large. To fill this gap, we used the sets $H_\varepsilon(\mu;\omega;\delta)$ and $G_\varepsilon(\mu;\omega;\delta)$ in Theorem \ref{t5} instead of set $G$ in Theorem $A$.
\end{remark}

Now, we will try to give the new versions of Theorems $A$ and $B$ using the sets $G_\varepsilon(\mu;\omega;\delta)$. The following theorem is the new version of Theorem A. However, again again in this new version,  the calibre (or equivalently the radius of the Bergman balls) must be small enough.

\begin{theorem}
Let $\omega\in \mathcal{R}$, $0<\varepsilon, C,q<\infty$, and $\mu$ be $\omega$-Carleson measure. Then, there is a calibre $\delta$  such that if $G_\varepsilon(\mu;\omega;\delta)$ is a $C$-dominated set for $A^q_\omega$, then $\mu$ is a reverse $\omega$-Carleson measure.
\end{theorem}
\begin{proof}
Consider $\alpha>0$ such that $\alpha q>1$. Then,
\begin{align*}
C\| f\|_{\omega,q}^q &\leq   \int_{G_\varepsilon(\mu;\omega;\delta)} |f|^q dV_\omega\\
 &\lesssim \sum_{Q\in H_\varepsilon(\mu;\omega;\delta)} V_\omega(B(c(Q),2R_\delta)). \langle |f|^\frac{1}{\alpha}\rangle_{B(c(Q),2R_\delta)}^{\alpha q}\\
&\lesssim \sum_{Q\in H_\varepsilon(\mu;\omega;\delta)} \mu(B(c(Q),2R_\delta)). \langle |f|^\frac{1}{\alpha}\rangle_{B(c(Q),2R_\delta)}^{\alpha q}\\
&\lesssim  \sum_{Q\in H_\varepsilon(\mu;\omega;\delta)} \mu(B(c(Q),2R_\delta)). \Big(\langle |f|^\frac{1}{\alpha}\rangle_{B(c(Q),2R_\delta),\mu} +   \langle |f - \langle |f|\rangle_{K_{c(Q)},\mu}|^\frac{1}{\alpha}\rangle_{K_{c(Q)}}  \Big)^{\alpha q}\\
&\lesssim  \sum_{Q\in H_\varepsilon(\mu;\omega;\delta)} \mu(K_{c(Q)}).\langle |f|^\frac{1}{\alpha}\rangle_{B(c(Q),2R_\delta),\mu}^{\alpha q}\\
& \qquad +  \sum_{Q\in H_\varepsilon(\mu;\omega;\delta)} V_\omega(B(c(Q),2R_\delta)).  \langle |f - \langle |f|\rangle_{B(c(Q),2R_\delta),\mu}|^\frac{1}{\alpha}\rangle_{B(c(Q),2R_\delta)}^{\alpha q}.
\end{align*}
By an standard proof we can see that $|f(z)-f(\zeta)|\lesssim \beta(z,\zeta) \|f\|_{\omega,q}$, when $\beta(z,\zeta)\leq 2 R_\delta$. Thus,
\begin{align*}
\| f\|_{\omega,q}^q &\lesssim \int_\mathbb{B} \mathcal{M}_{\mu, \mathfrak{D}_\delta} (|f|^\frac{1}{\alpha})(z)^{\alpha q} d\mu(z) +R_\delta \| f\|_{\omega,q}^q.
\end{align*}
Now, from the boundedness of $\mathcal{M}_{\mu,\mathfrak{D}_\delta}: L_\mu^{\alpha q}\rightarrow L_\mu^{\alpha q}$ and by taking $\delta$ small enough, we conclude the desired result.
\end{proof}

The following Theorem is a new version of Theorem $B$.  Note that in part (i) of this new version, the radius of the Bergman balls does not need to be large.

\begin{theorem} \label{t3}
Let $\omega\in \mathcal{R}$ and $0<q<\infty$. If $\mu$ is a reverse $\omega$-Bergman Carleson measure,
then the following statements hold.
\begin{itemize}
\item[(i)]  For every calibre $\delta$, there is an $\varepsilon>0$ such that $G_\varepsilon(\mu;\omega;\delta)$ is a dominated set for $A_\omega^q(\mathbb{B})$.
\item[(ii)] There are $\delta,\varepsilon>0$ such that  $H_\varepsilon(\mu;\omega;\delta)=\mathfrak{D}_\delta$.
\end{itemize}
\end{theorem}
\begin{proof}
(i). Fix $\delta>0$. Let $\delta_0<\delta$ be such that $2R_{\delta_0}\leq R_{\delta}$.  By using Theorem \ref{t5} (iv), there is an $\varepsilon_0>0$ such that
$$\| f\|_{\omega,q}^q \lesssim  \sum_{Q\in H_{\varepsilon_0}(\mu;\omega;\delta_0)} V_\omega(B(c(Q),R_{\delta_0})). \langle |f|^{q-r}\rangle_{B(c(Q),R_{\delta_0})}. \langle |f|^r\rangle_{B(c(Q),R_{\delta_0})} \mathbf{1}_{H_{\varepsilon_0}(f;\mu;\delta_0;r)} (Q).$$
It is not difficult  to see that there is an $\varepsilon>0$ such that
$$H_{\varepsilon_0}(\mu;\omega;\delta_0)\subset H_{\varepsilon}(\mu;\omega;\delta).$$
Thus, from the above inclusion and (\ref{e12}), we deduce that
$$\| f\|_{\omega,q}^q \lesssim  \sum_{Q\in H_{\varepsilon}(\mu;\omega;\delta)} V_\omega(B(c(Q),R_{\delta})). \langle |f|^\frac{1}{\alpha}\rangle_{B(c(Q),R_{\delta})}^{\alpha(q-r)}. \langle |f|^\frac{1}{\alpha}\rangle_{B(c(Q),R_{\delta})}^{\alpha r}.$$
Now, from (\ref{e36}) and the fact that $B(c(Q),R_{\delta})=B_{c(Q)}\subseteq  G_\varepsilon(\mu;\omega;\delta)$, for every $Q\in H_\varepsilon(\mu;\omega;\delta)$, we obtain
\begin{align*}
\| f\|_{\omega,q}^q &\lesssim \sum_{Q\in H_\varepsilon(\mu;\omega;\delta)} V_\omega(B_{c(Q)}). \langle |f|^\frac{1}{\alpha}\rangle_{B_{c(Q)}}^{\alpha(q-r)}. \langle |f|^\frac{1}{\alpha}\rangle_{B_{c(Q)}}^{\alpha r}\\
 &=  \sum_{Q\in H_\varepsilon(\mu;\omega;\delta)} V_\omega(B_{c(Q)}). \langle |\mathbf{1}_{G_\varepsilon(\mu;\omega;\delta)} f|^\frac{1}{\alpha}\rangle_{B_{c(Q)}}^{\alpha q}\\
 &\lesssim \int_\mathbb{B} \mathcal{M}_{\omega, \mathfrak{D}_\delta} (|\mathbf{1}_{G_\varepsilon(\mu;\omega;\delta)} f|^{\frac{1}{\alpha}})^{\alpha q} dV_\omega\\
&\lesssim \int_\mathbb{B} |\mathbf{1}_{G_\varepsilon(\mu;\omega;\delta)} f_k|^q   dV_\omega=\int_{G_\varepsilon(\mu;\omega;\delta)} |f|^q dV_\omega.
\end{align*}
The last inequality follows from the boundedness of $\mathcal{M}_{\omega, \mathfrak{D}_\delta}:L_\omega^{\alpha q}\rightarrow L_\omega^{\alpha q}$.

(ii) follows by a similar argument as in the proof of \cite[Theorem 4.3]{luecking5}. Hence, we omitted its proof.
\end{proof}

The following theorem and Theorem \ref{t5} show that the gap between the property that a measure $\mu$ is a reverse $\omega$-Bergman Carleson measure and the property that the set $G_\varepsilon(\mu;\omega;\delta)$ is a dominated set for $A^q_\omega$ is that $\mathcal{M}_{\mu,\mathfrak{D}_\delta}$ or $\mathcal{M}_{\omega,\mathfrak{D}_\delta}$  from $A_\omega^q$ to $L_\mu^q$ is bounded below, respectively.

\begin{theorem}
Let $\omega\in \mathcal{R}$ and $0<q<\infty$. Let $\alpha\geq 1$ be such that $\alpha q>1$. If $\mu$ is $\omega$-Bergman Carleson measure,
then the following statements are equivalent.
\begin{itemize}
\item[(i)]  $\mathcal{M}_{\omega,\mathcal{T}}(|.|^\frac{1}{\alpha})^\alpha: A_\omega^q\rightarrow L_\mu^q$ is bounded below.
\item[(ii)]  For some (every) calibre $\delta$, there is an $\varepsilon>0$ such that $G_\varepsilon(\mu;\omega;\delta)$ is a dominated set for $A^q_\omega$.
\end{itemize}
\end{theorem}
\begin{proof}
(i)$\Rightarrow$(ii) is deduced by a similar argument as in the proof of Theorem \ref{t5} [(i)$\Rightarrow$(iv)].

(ii)$\Rightarrow$(i). Using Theorem \ref{t},
\begin{align*}
\| f\|_{\omega,q}^q &\lesssim \int_{G_\varepsilon(\mu;\omega;\delta)} |f|^q dV_\omega= \sum_{c(Q)\in H_\varepsilon(\mu;\omega;\delta)}  V_\omega(B_{c(Q)}). \langle |f|^{q}\rangle_{B_{c(Q)}} \\
&\lesssim \sum_{c(Q)\in H_\varepsilon(\mu;\omega;\delta)} \mu(B_{c(Q)}). \langle |f|^{q}\rangle_{B_{c(Q)}}\leq \int_\mathbb{B} \mathcal{M}_{\omega, \mathfrak{D}_\delta} (|f|^{\frac{1}{\alpha}})^{\alpha q} d\mu(z) \\
&= \|\mathcal{M}_{\omega, \mathfrak{D}_\delta} (|f|^{\frac{1}{\alpha}})^{\alpha}\|_{q,\mu}^q.
\end{align*}
This is the desired result.
\end{proof}

\section{Appendixes}

\subsection{Proof of Lemma \ref{l5}}
It suffices to prove that $\mathcal{M}_{t,\mu,\mathcal{T}}:L^p_\mu\rightarrow L^{q,\infty}_\mu$ is bounded. Then, the result follows from the interpolating theorem.

Let $f\in L^p_\mu$ and $s>0$. Define
$$O_s=\{ z\in \mathbb{B}; \ \mathcal{M}_{t,\mu,\mathcal{T}}(|f|)>s\}.$$
Let $\Gamma$ be the set of dyadic cubes $Q$ such that
$$ \mu (\widehat{K_{c(Q)}})^t \langle |f|\rangle_{\mu, \widehat{K_{c(Q)}}}>s,$$
and let $\Gamma^{\max}$ be the subfamily of $\Gamma$ consisting of cubes $Q$ such that $\widehat{K_{c(Q)}}$ be a maximal tent.
Clearly, $\bigcup_{Q\in \Gamma^{\max}} \widehat{K_{c(Q)}}$ is a covering of $O_s$ and each $z\in O_s$ is at most in two tents $\widehat{K_{c(Q)}}$, where $Q\in\Gamma^{\max}$. Thus,
\begin{align} \label{e18}
s^q \mu (O_s)&\leq \sum_{Q \in \Gamma^{\max}} \mu(\widehat{K_{c(Q)}}) s^q \notag\\
&\leq  \sum_{Q \in \Gamma^{\max}} \mu(\widehat{K_{c(Q)}})  \Big(\mu (\widehat{K_{c(Q)}})^t \langle |f|\rangle_{\mu, \widehat{K_{c(Q)}}}\Big)^q\notag\\
&= \sum_{Q \in \Gamma^{\max}} \mu(\widehat{K_{c(Q)}})^{1+sq-q} \Big( \int_{\widehat{K_{c(Q)}}} |f|d\mu \Big)^q.
\end{align}
By H\"{o}lder's inequality,
\begin{align*}
\ref{e18}&\leq \sum_{Q \in \Gamma^{\max}} \mu(\widehat{K_{c(Q)}})^{1+sq-q+\frac{q}{p^\prime}} \Big( \int_{\widehat{K_{c(Q)}}} |f|^p d\mu \Big)^\frac{q}{p}\\
&= \sum_{Q \in \Gamma^{\max}} \Big( \int_{\widehat{K_{c(Q)}}} |f|^p d\mu \Big)^\frac{q}{p} \leq \Big( \sum_{Q \in \Gamma^{\max}} \int_{\widehat{K_{c(Q)}}} |f|^p d\mu \Big)^\frac{q}{p}\\
&\leq \Big( 2 \int_{\mathbb{B}} |f|^p d\mu \Big)^\frac{q}{p}
\end{align*}
The proof is complete.

\subsection{Generalizing the paper to other domains}
It can be observed that the results of this paper can be extended to the weighted Bergman spaces defined on upper half-plane and strictly pseudo-convex bounded domains.
\begin{itemize}
\item[(i)] Let $\mathbb{R}^{2+}$ be the upper half-plane. Then, the weighted Bergman space $A^q_\alpha(\mathbb{R}^{2+})$ is defined by the weight $\omega_\alpha(z)= \frac{1}{\pi} (\alpha +1) (2 Im \ z)^\alpha$ on $\mathbb{R}^{2+}$. For the details and needed preliminaries obtained on these spaces see \cite{wick2}.

\item[(ii)] Let $\Omega$ be a strictly pseudo-convex bounded domain in $\mathbb{C}^n$ and let $\delta(z)$ be the distance of $z$ from the boundary of $\Omega$.  Then, the weighted Bergman space $A^q_\alpha(\Omega)$ is defined by the weight $\omega_\alpha(z)= \delta(z)^\alpha$. All the needed basic results for the stated extension may be found in \cite{hu}.
\end{itemize}

\section*{Acknowledgments}
The idea of working on this project came to the author from Dr. Mahdi Hormozi in Jan. or Feb. 2020.
The author would like to thank him for his
valuable comments and advice which certainly improved the paper.

\vspace*{0.5cm}

 Hamzeh Keshavarzi

E-mail: Hamzehkeshavarzi67@gmail.com

Department of Mathematics, College of Sciences,
Shiraz University, Shiraz, Iran

\end{document}